\documentclass[a4paper,12pt]{amsart}

\usepackage{amsmath}
\usepackage{amssymb}
\usepackage{mathrsfs}
\usepackage{enumerate}
\usepackage{ifthen}
\usepackage{graphicx}
\usepackage[T1]{fontenc} %skandit

\nonstopmode \numberwithin{equation}{section}
\newtheorem{thm}{Theorem}[section]
\newtheorem{cor}{Corollary}[section]
\newtheorem{lem}{Lemma}[section]
\newtheorem{prop}[equation]{Proposition}

\newtheorem{conj}{Conjecture}

\theoremstyle{definition}

\newtheorem{prob}{Problem}[section]
\newtheorem{rem}{Remark}[section]

\newcounter{minutes}
\divide\time by 60
\newcounter{hours}
\multiply\time by 60
\addtocounter{minutes}{-\time}
\newcounter {own}
\def\theown {\thesection       .\arabic{own}}

\newenvironment{pf}[1][]{%
 \vskip 3mm
 \noindent
 \ifthenelse{\equal{#1}{}}%
  {{\slshape Proof. }}%
  {{\slshape #1.} }%
 }%
{\qed\bigskip}

\newcounter{alphabet}
\newcounter{tmp}

\def\be{\begin{equation}}
\def\ee{\end{equation}}

\newcommand{\bee}{\begin{enumerate}}
\newcommand{\eee}{\end{enumerate}}

\newcommand{\blem}{\begin{lem}}
\newcommand{\elem}{\end{lem}}
\newcommand{\bthm}{\begin{thm}}
\newcommand{\ethm}{\end{thm}}
\newcommand{\bcor}{\begin{cor}}
\newcommand{\ecor}{\end{cor}}
\newcommand{\beg}{\begin{examp}}
\newcommand{\eeg}{\end{examp}}
\newcommand{\begs}{\begin{examples}}
\newcommand{\eegs}{\end{examples}}
\newcommand{\bdefe}{\begin{defin}}
\newcommand{\edefe}{\end{defin}}
\newcommand{\bprob}{\begin{prob}}
\newcommand{\eprob}{\end{prob}}
\newcommand{\bei}{\begin{itemize}}
\newcommand{\eei}{\end{itemize}}

\newcommand{\bcon}{\begin{conj}}
\newcommand{\econ}{\end{conj}}
\newcommand{\bcons}{\begin{conjs}}
\newcommand{\econs}{\end{conjs}}
\newcommand{\bprop}{\begin{prop}}
\newcommand{\eprop}{\end{prop}}
\newcommand{\br}{\begin{rem}}
\newcommand{\er}{\end{rem}}
\newcommand{\brs}{\begin{rems}}
\newcommand{\ers}{\end{rems}}
\newcommand{\bo}{\begin{obser}}
\newcommand{\eo}{\end{obser}}
\newcommand{\bos}{\begin{obsers}}
\newcommand{\eos}{\end{obsers}}
\newcommand{\bpf}{\begin{pf}}
\newcommand{\epf}{\end{pf}}
\newcommand{\ba}{\begin{array}}
\newcommand{\ea}{\end{array}}
\newcommand{\beq}{\begin{eqnarray}}
\newcommand{\beqq}{\begin{eqnarray*}}
\newcommand{\eeq}{\end{eqnarray}}
\newcommand{\eeqq}{\end{eqnarray*}}

\begin{document}

\title{Logarithmic coefficients of some close-to-convex functions}

\author{Md Firoz Ali}
\address{Md Firoz Ali,
Department of Mathematics,
Indian Institute of Technology Kharagpur,
Kharagpur-721 302, West Bengal, India.}
\email{ali.firoz89@gmail.com}

\author{A. Vasudevarao}
\address{A. Vasudevarao,
Department of Mathematics,
Indian Institute of Technology Kharagpur,
Kharagpur-721 302, West Bengal, India.}
\email{alluvasu@maths.iitkgp.ernet.in}

\subjclass[2010]{Primary 30C45, 30C55}
\keywords{Analytic, univalent, starlike, convex, close-to-convex functions, logarithmic coefficient.}

\def\thefootnote{}
\footnotetext{ {\tiny File:~\jobname.tex,
printed: \number\year-\number\month-\number\day,
          \thehours.\ifnum\theminutes<10{0}\fi\theminutes }
} \makeatletter\def\thefootnote{\@arabic\c@footnote}\makeatother

\begin{abstract}
The logarithmic coefficients $\gamma_n$ of an analytic and univalent function $f$ in the unit disk $\mathbb{D}=\{z\in\mathbb{C}:|z|<1\}$ with the normalization $f(0)=0=f'(0)-1$ are defined by $\log \frac{f(z)}{z}= 2\sum_{n=1}^{\infty} \gamma_n z^n$. In the present paper, we consider close-to-convex functions (with argument $0$) with respect to odd starlike functions and determine the sharp upper bound of $|\gamma_n|$, $n=1,2,3$ for such functions $f$.
\end{abstract}

\thanks{}

\maketitle
\pagestyle{myheadings}
\markboth{Md Firoz Ali and A. Vasudevarao }{Logarithmic coefficients of some close-to-convex functions}

\section{Introduction}

Let $\mathcal{A}$ denote the class of analytic functions $f$ in the unit disk $\mathbb{D}=\{z\in\mathbb{C}:|z|<1\}$ normalized by $f(0)=0=f'(0)-1$.
Any function $f$ in $\mathcal{A}$ has the following power series representation
\begin{equation}\label{p7-001}
f(z)= z+\sum_{n=2}^{\infty}a_n z^n.
\end{equation}
The class of univalent (i.e. one-to-one) functions in $\mathcal{A}$ is denoted by $\mathcal{S}$. A function $f\in\mathcal{A}$ is called starlike (convex respectively) if $f(\mathbb{D})$ is starlike with respect to the origin (convex respectively). Let $\mathcal{S}^*$ and $\mathcal{C}$ denote the class of starlike and convex functions in $\mathcal{S}$ respectively. It is well-known that a function $f\in\mathcal{A}$ is in $\mathcal{S}^*$ if and only if ${\rm Re\,}\left(zf'(z)/f(z)\right)>0$ for $z\in\mathbb{D}$. Similarly, a function $f\in\mathcal{A}$ is in $\mathcal{C}$ if and only if ${\rm Re\,}\left(1+zf''(z)/f'(z)\right)>0$ for $z\in\mathbb{D}$. From the above it is easy to see that $f\in\mathcal{C}$ if and only if $zf'\in\mathcal{S}^*$. Given $\alpha\in(-\pi/2,\pi/2)$ and $g\in\mathcal{S}^*$, a function $f\in\mathcal{A}$ is said to be close-to-convex with argument $\alpha$ and with respect to $g$ if
\begin{equation}\label{p7-003}
{\rm Re\,} \left(e^{i\alpha}\frac{zf'(z)}{g(z)}\right)>0 \quad\mbox{ for } z\in\mathbb{D}.
\end{equation}
Let $\mathcal{K}_{\alpha}(g)$ denote the class of all such functions. Let
$$
\mathcal{K}(g):= \bigcup_{\alpha\in(-\pi/2,\pi/2)} \mathcal{K}_{\alpha}(g) \quad\mbox{ and }\quad \mathcal{K}_{\alpha}:= \bigcup_{g\in\mathcal{S}^*} \mathcal{K}_{\alpha}(g)
$$
be the classes of functions called close-to-convex functions with respect to $g$ and close-to-convex functions with
argument $\alpha$, respectively. The class
$$
\mathcal{K}:= \bigcup_{\alpha\in(-\pi/2,\pi/2)} \mathcal{K}_{\alpha}= \bigcup_{g\in\mathcal{S}^*} \mathcal{K}(g)
$$
is the class of all close-to-convex functions. It is well-known that every close-to-convex function is univalent in $\mathbb{D}$ (see \cite{Duren-book}). Geometrically, $f\in\mathcal{K}$ means that the complement of the image-domain $f(\mathbb{D})$ is the union of non-intersecting half-lines.
%of rays that are disjoint (except that the origin of one ray may lie on another one of the rays).

For the function $f\in\mathcal{S}$,  the logarithmic coefficients $\gamma_n (n=1,2,\ldots)$ are defined by
\begin{equation}\label{p7-005}
\log \frac{f(z)}{z}= 2\sum_{n=1}^{\infty} \gamma_n z^n, \quad z\in\mathbb{D}.
\end{equation}
%Logarithmic coefficients are of interest to study as by their terms it is convenient for
%applications to write restrictions where the univalence property is efficiently revealed.
I.E. Bazilevich first noticed that the logarithmic coefficients are very essential in the coefficient problem of univalent functions. He estimated (see \cite{Bazilevich-1965,Bazilevich-1967}) in depending upon the positive Hayman constant (see \cite{Hayman-book}) how close the coefficients $\gamma_n$ $(n=1,2,\ldots)$ of the functions of class $\mathcal{S}$ are to the relative logarithmic coefficients of the Koebe function $k(z)=z/(1-z)^2$. He also estimated the value $\sum_{n=1}^{\infty} n|\gamma_n|^2r^{2n}$ which after multiplication by $\pi$ is equal to the area of the image of the disk $|z|<r<1$ under the function $\frac{1}{2} \log(f(z)/z)$ for  $f\in\mathcal{S}$. The celebrated de Branges' inequalities (the former Milin conjecture) for univalent
functions $f$ state that
$$
\sum_{k=1}^{n} (n-k+1)|\gamma_k|^2\le \sum_{k=1}^{n} \frac{n+1-k}{k},\quad n=1,2,\ldots,
$$
with equality if and only if $f(z)=e^{-i\theta}k(e^{i\theta}z), \theta\in\mathbb{R}$ (see \cite{Branges-1985}). De Branges \cite{Branges-1985} used this inequality to prove the celebrated Bieberbach conjecture. Moreover,
the de Branges' inequalities have also  been  the source of many other interesting inequalities involving logarithmic coefficients of $f\in\mathcal{S}$ such as (see \cite{Duren-Leung-1979})
$$
\sum_{k=1}^{\infty} |\gamma_k|^2\le \sum_{k=1}^{\infty} \frac{1}{k^2} =\frac{\pi^2}{6}.
$$
%The logarithmic coefficients $\gamma_n$ play a central role in the theory of univalent functions. Very few exact upper bounds for $\gamma_n$ seem have been established. The significance of this problem in the context of Bieberbach conjecture was pointed out by Milin in his conjecture. Milin conjectured that for $f\in\mathcal{S}$ and $n\ge 2$,
%$$
%\sum_{m=1}^{n}\sum_{k=1}^{m} \left(k|\gamma_k|^2-\frac{1}{k}\right)\le 0,
%$$
%which led De Branges, by proving this conjecture, to the proof of the Bieberbach conjecture \cite{Branges-1985}.
More attention has been given to the results of an average sense (see \cite{Duren-book,Duren-Leung-1979,Roth-2007}) than the exact upper bounds for $|\gamma_n|$ for functions in the class $\mathcal{S}$.
Very few exact upper bounds for $|\gamma_n|$ seem have been established. For the Koebe function $k(z)=z/(1-z)^2$, the logarithmic coefficients are $\gamma_n=1/n$. Since the Koebe function $k(z)$ plays the role of extremal function for most of the extremal problems in the class $\mathcal{S}$, it is expected that $|\gamma_n|\le \frac{1}{n}$ holds for functions in $\mathcal{S}$. But this is not true in general, even in order of magnitude \cite[Theorem 8.4]{Duren-book}. Indeed, there exists a bounded function $f$ in the class $\mathcal{S}$ with logarithmic coefficients $\gamma_n\ne O(n^{-0.83})$ (see \cite[Theorem 8.4]{Duren-book}).

By differentiating (\ref{p7-005}) and equating coefficients we obtain
\begin{eqnarray}
\gamma_1&=&\frac{1}{2} a_2\label{p7-010}\\
\gamma_2&=&\frac{1}{2}(a_3-\frac{1}{2}a_2^2)\label{p7-015}\\
\gamma_3&=&\frac{1}{2}(a_4-a_2a_3+\frac{1}{3}a_2^3)\label{p7-020}.
\end{eqnarray}
If $f\in\mathcal{S}$ then $|\gamma_1|\le 1$ follows from (\ref{p7-010}). Using Fekete-Szeg\"{o} inequality \cite[Theorem 3.8]{Duren-book} in (\ref{p7-015}), it is easy to obtain the following sharp estimate
$$
|\gamma_2|\le \frac{1}{2}(1+2e^{-2})=0.635\ldots.
$$
For $n\ge 3$, the problem seems much harder, and no significant upper bound for $|\gamma_n|$ when $f\in\mathcal{S}$ appear to be known.

For functions in the class $\mathcal{S}^*$, using the analytic characterization ${\rm Re\,}\left(zf'(z)/f(z)\right)>0$ for $z\in\mathbb{D}$ it is easy to prove that $|\gamma_n|\le \frac{1}{n}$ for $n\ge 1$ and equality holds for the Koebe function $k(z)=z/(1-z)^2$.  The inequality $|\gamma_n|\le \frac{1}{n}$ for $n\ge 2$ also holds for functions in the class $\mathcal{K}$ was claimed in a paper of Elhosh \cite{Elhosh-1996}. However, Girela \cite{Girela-2000} pointed out an error in the proof of Elhosh \cite{Elhosh-1996} and, hence, the result is not substantiated. Indeed, Girela proved that for each $n\ge 2$, there exists a function $f\in\mathcal{K}$ such that $|\gamma_n|> \frac{1}{n}$. Recently, it has been proved \cite{Thomas-2016} that $|\gamma_3|\le \frac{7}{12}$ for functions in $\mathcal{K}_0$ (close-to-convex functions with argument $0$) with the additional assumption that the second coefficient of the corresponding starlike function $g$ is real. But this estimate is not sharp as pointed out in \cite{Ali-Vasudevarao-2016} where the authors proved that $|\gamma_3|\le \frac{1}{18} (3+4 \sqrt{2})=0.4809$ for functions in $\mathcal{K}_0$ without assuming the additional assumption that the second coefficient of the corresponding starlike function $g$ is real. In the same paper, the authors also determined the sharp upper bound $|\gamma_3|\le \frac{1}{243} (28+19 \sqrt{19})=0.4560$ for close-to-convex functions with argument $0$ and with respect to the Koebe function and conjectured that this upper bound is also true for the whole class $\mathcal{K}_0$.
%Thomas claimed that this estimate is sharp and has given a form of the extremal function. But after rigorous reading of the paper \cite{Thomas-2016}, we observed that such functions do not belong to the class $\mathcal{K}_0$ (more details will be given in Section \ref{Main Results}).

%By  fixing a  starlike function $g$ in the class $\mathcal{S}^*$, the inequality (\ref{p7-003}) gives a specific  subclass of close-to-convex functions.
%One of such important  subclasses is the class of close-to-convex functions with respect to the Koebe function $k(z)=z/(1-z)^2$ which has been considered in \cite{Ali-Vasudevarao-2016}.
Let $\mathcal{S}_2^*$  denote the class of odd starlike functions and $\mathcal{F}$
denote  the class of close-to-convex functions with argument $0$ and with respect to odd starlike functions. That is,
$$
\mathcal{F}=\left\{f\in\mathcal{A}: {\rm Re\,} \frac{zf'(z)}{g(z)}>0, ~~ z\in\mathbb{D}, \mbox{ for some } g\in\mathcal{S}_2^* \right\}.
$$
It is important to note that the class $\mathcal{F}$ is rotationally invariant. In the present article, we determine the sharp upper bound of $|\gamma_n|$, $n=1,2,3$ for functions in $\mathcal{F}$.

\section{Main Results}\label{Main Results}

Let $\mathcal{P}$ denote the class of analytic functions $P$ of the form
\begin{equation}\label{p7-025}
P(z)= 1+\sum_{n=1}^{\infty}c_n z^n
\end{equation}
such that  ${\rm\, Re\,} P(z)>0$ in $\mathbb{D}$.
Functions in $\mathcal{P}$ are sometimes called Carath\'{e}odory functions. To prove our main results, we need some preliminary lemmas.
%The first one is known as Carath\'{e}odory's lemma (see \cite[p.41]{Duren-book} for example) and the second one is due to Libera and Z{\l}otkiewicz \cite{Libera-Zlotkiewicz-1982}.

\begin{lem}\cite[p. 41]{Duren-book}\label{p7-lemma001}
For a function $P\in\mathcal{P}$ of the form (\ref{p7-025}), the sharp inequality $|c_n|\le 2$ holds for each $n\ge 1$. Equality holds for the function $P(z)=(1+z)/(1-z)$.

\end{lem}

\begin{lem}\cite{Ma-Minda-1992}\label{p7-lemma003}
Let $P\in\mathcal{P}$ be of the form (\ref{p7-025}) and $\mu$ be a complex number. Then
\begin{equation}\label{firoz_vasu_p1_i030}
|c_2-\mu c_1^2|\le 2\, \max\{1,|2\mu-1|\}.
\end{equation}
The result is sharp for the functions given by $P(z)=\frac{1+z^2}{1-z^2}$ and $P(z)=\frac{1+z}{1-z}$.
\end{lem}

%We also required the following result due to Carath\'{e}odory and Toeplitz (see \cite{Grenander-Szego-1958} or, \cite{Tsuji-1959}).
%
%\begin{lem}[Carath\'{e}odory-Toeplitz theorem]\label{p7-lemma005}
%The power series $P(z)= 1+\sum_{n=1}^{\infty}c_n z^n$ converges in the open unit disk $\mathbb{D}$ to a function in the Carath\'{e}odory class $\mathcal{P}$ if and only if the Toeplitz determinant
%$$
%D_n=
%\begin{vmatrix}
% 2 & c_1 & c_2 & \cdots & c_n\\[2mm]
% c_{-1} & 2 & c_1 & \cdots & c_{n-1}\\
%  \vdots & \vdots  & \vdots  & \vdots & \vdots  \\
%c_{-n} & c_{-n+1} & c_{-n+2} & \cdots & 2
%  \end{vmatrix}
%$$
%is non-negative for every $n\ge 1$, where $c_{-k}=\overline{c_k}$. There are exactly two possible cases. Either $D_n>0$ for each $n\ge 1$, or $D_n>0$ for $n\le m-1$ and $D_n=0$ for $n\ge m$. In the later case, $P(z)$ is of the following form
%$$
%P(z)= \sum_{k=1}^{m} \lambda_k \frac{1+x_kz}{1-x_kz},
%$$
%where $\lambda_k>0$ with $\sum_{k=1}^{m} \lambda_k=1$ and $|x_k|=1$ with $x_k\ne x_j$ for $k\ne j$.
%\end{lem}

%By making use of the Carath\'{e}odory-Toeplitz theorem, Libera and Z{\l}otkiewicz \cite{Libera-Zlotkiewicz-1982} proved the following result.

\begin{lem}\cite{Libera-Zlotkiewicz-1982}\label{p7-lemma010}
Let $P\in\mathcal{P}$ be of the form (\ref{p7-025}). Then there exist $x, t\in\mathbb{C}$ with $|x|\le 1$ and $|t|\le 1$ such that
\begin{eqnarray*}
2c_2&=& c_1^2 + x(4 - c_1^2)\quad\mbox { and }\\
4c_3&=& c_1^3+ 2(4-c_1^2)c_1x-c_1(4-c_1^2)x^2+2(4-c_1^2)(1-|x|^2)t.
\end{eqnarray*}
\end{lem}

%In \cite{Thomas-2016},  Thomas claimed that his result (i.e. $|\gamma_3|\le 7/12$) is sharp for functions in the class $\mathcal{K}_0$
%by ascertaining the equality holds for a function $f$ defined by $zf'(z)=g(z)P(z)$ where $g\in\mathcal{S}^*$ with $b_2(g)=b_3(g)=b_4(g)=2$ and $P\in\mathcal{P}$ with $c_1(P)=0$, $c_2(P)=c_3(P)=2$. But in view of  Lemma \ref{p7-lemma010}, it is easy to see that there does not exist a function $P\in\mathcal{P}$ with the property $c_1(P)=0$, $c_2(P)=c_3(P)=2$. Thus we can conclude that the result obtained by Thomas is not sharp. This was first pointed out by the authors in \cite{Ali-Vasudevarao-2016}, where the authors also improved the bound for $|\gamma_3|$ for functions in the class $\mathcal{K}_0$.

\begin{thm}\label{p7-theorem-001}
Let $f\in\mathcal{F}$ be of the form (\ref{p7-001}). Then $|\gamma_1|\le \frac{1}{2}$, $|\gamma_2|\le \frac{1}{2}$ and $|\gamma_3|\le \frac{1}{972}(95+23 \sqrt{46})$. The inequalities are sharp.
\end{thm}

\begin{proof}
Let $f\in\mathcal{F}$ be of the form (\ref{p7-001}). Then there exists an odd starlike function $g(z)=z+\sum_{n=1}^{\infty}b_{2n+1} z^{2n+1}$ and a Carath\'{e}odory function $P\in\mathcal{P}$ of the form (\ref{p7-025}) such that
\begin{equation}\label{p7-029a}
zf'(z)=g(z)P(z).
\end{equation}
Comparing the coefficients on the both sides of (\ref{p7-029a}) gives
\begin{equation}\label{p7-029abb}
a_2=\frac{1}{2}c_1,~~ a_3=\frac{1}{3}(b_3+c_2)~\mbox { and }~~ a_4=\frac{1}{4}(b_3c_1+c_3).
\end{equation}
Substituting $a_2, a_3$ and $a_4$ given by (\ref{p7-029abb}) in (\ref{p7-010}), (\ref{p7-015}) and (\ref{p7-020}) and then further simplification gives
\begin{eqnarray}
\gamma_1 &=& \frac{1}{2}a_2 =\frac{1}{4}c_1\label{p7-030a}\\[1mm]
\gamma_2 &=& \frac{1}{2}\left(a_3-\frac{1}{2}a_2^2\right) =\frac{1}{6}b_3+\frac{1}{6}\left(c_2-\frac{3}{8}c_1^2\right)\label{p7-030b}\\[1mm]
2\gamma_3 &=& a_4-a_2a_3+\frac{1}{3}a_2^3 =\frac{1}{24}\left(2c_1b_3+c_1^3-4c_1c_2+6c_3\right)\label{p7-030}.
\end{eqnarray}
%\begin{equation}\label{p7-030a}
%\gamma_1 = \frac{1}{2}a_2 =\frac{1}{4}c_1,
%\end{equation}
%\begin{equation}\label{p7-030b}
%\gamma_2 = \frac{1}{2}(a_3-\frac{1}{2}a_2^2) =\frac{1}{6}b_3+\frac{1}{6}(c_2-\frac{3}{8}c_1^2)
%\end{equation}
%\begin{equation}\label{p7-030}
%2\gamma_3 = a_4-a_2a_3+\frac{1}{3}a_2^3 =\frac{1}{24}\left(2c_1b_3+c_1^3-4c_1c_2+6c_3\right)
%\end{equation}
 In view of Lemma \ref{p7-lemma001}, it follows from (\ref{p7-030a}) that $|\gamma_1|\le \frac{1}{2}$ and equality holds for a function $f$ defined by $zf'(z)=g(z)P(z)$, where $g(z)=z/(1-z^2)$ and $P(z)=(1+z)/(1-z)$. Since $g$ is an odd starlike function, $|b_3|\le 1$ (see \cite[Chaptar 4, Theorem 3, page 35]{Goodman-book}). Using Lemma \ref{p7-lemma003}, it follows from (\ref{p7-030b}) that
$$
|\gamma_2|\le \frac{1}{6}|b_3|+\frac{1}{6}\left|c_2-\frac{3}{8}c_1^2\right| \le \frac{1}{6}+\frac{1}{3}=\frac{1}{2}
$$
and equality holds for a function $f$ defined by $zf'(z)=g(z)P(z)$, where $g(z)=z/(1-z^2)$ and $P(z)=(1+z^2)/(1-z^2)$.

Writing $c_2$ and $c_3$ in terms of $c_1$ with the help of Lemma \ref{p7-lemma010}, it follows from (\ref{p7-030}) that
\begin{equation}\label{p7-035}
48\gamma_3= 2c_1b_3 +\frac{1}{2}c_1^3+c_1x(4-c_1^2)-\frac{3}{2}c_1x^2(4-c_1^2)+3(4-c_1^2)(1-|x|^2)t,
\end{equation}
where $|x|\le 1$ and $|t|\le 1$. Since the class $\mathcal{F}$ is invariant under rotation, without loss of generality we can assume that $c_1=c$, where $0\le c\le 2$. Taking modulus on both the sides of (\ref{p7-035}) and then applying triangle inequality and  $|b_3|\le 1$, it follows that
$$
48|\gamma_3|\le 2c+\left|\frac{1}{2}c^3+cx(4-c^2)-\frac{3}{2}cx^2(4-c^2)\right|+3(4-c^2)(1-|x|^2),
$$
where we have also used the fact $|t|\le 1$. Let $x=re^{i\theta}$ where $0\le r\le 1$ and $0\le\theta\le 2\pi$. For simplicity, by writing $\cos\theta=p$ we obtain
\begin{equation}\label{p7-042}
48|\gamma_3|\le \psi(c,r)+\left|\phi(c,r,p)\right|=:F(c,r,p),
\end{equation}
where $\psi(c,r)=2c+ 3(4-c^2)(1-r^2)$ and
\begin{align*}
\hspace{-1cm} \phi(c,r,p)
&=\left(\frac{1}{4}c^6+c^2r^2(4-c^2)^2+\frac{9}{4}c^2r^4(4-c^2)^2+c^4(4-c^2)rp\right.\\
&\qquad\quad \left.-\frac{3}{2}c^4r^2(4-c^2)(2p^2-1)-3c^2(4-c^2)r^3p \right)^{1/2}.
\end{align*}
Thus we need to find the maximum value of $F(c,r,p)$ over the rectangular cube $R:=[0,2]\times[0,1]\times[-1,1]$.

By elementary calculus it is  easy to verify the following:
\begin{align*}
&\max_{0\le r\le 1} \psi(0,r)=\psi\left(0,0\right)=12,\quad \max_{0\le r\le 1} \psi(2,r)=4,\\[2mm]
&\max_{0\le c\le 2} \psi(c,0)=\psi\left(\frac{1}{3},0\right)=\frac{37}{3},\quad \max_{0\le c\le 2} \psi(c,1)=\psi(2,1)=4 \quad\mbox { and }\\[2mm]
&\max_{(c,r)\in[0,2]\times[0,1]} \psi(c,r)=\psi\left(\frac{1}{3},0\right)=\frac{37}{3}.
\end{align*}
We first find the maximum value of $F(c,r,p)$ on the boundary of $R$, i.e on the six faces of the rectangular cube $R$.

On the face $c=0$, we have $F(0,r,p)=\psi(0,r)$ for  $(r,p)\in R_1:=[0,1]\times[-1,1]$. Thus
$$
\max_{(r,p)\in R_1} F(0,r,p)= \max_{0\le r\le 1} \psi(0,r)=\psi\left(0,0\right)=12.
$$

On the face $c=2$, we have $F(2,r,p)= 8$ for  $(r,p)\in R_1$.

On the face $r=0$, we have $F(c,0,p)= 2c+3(4-c^2)+\frac{1}{2}c^3$ for  $(c,p)\in R_2:=[0,2]\times[-1,1]$. Note that $F(c,0,p)$ is independent of $p$. Thus, by using elementary calculus it is easy to see that
$$
\max_{(c,p)\in R_2} F(c,0,p)= F\left(\frac{2}{3} (3-\sqrt{6}),0,p\right)=\frac{8}{9} \left(9+\sqrt{6}\right)=12.3546.
$$

On the face $r=1$, we have $F(c,1,p)=\psi(c,1)+ |\phi(c,1,p)|$ for  $(c,p)\in R_2$. We first prove that $\phi(c,1,p)\ne 0$ in the interior of $R_2$.
On the contrary, if $\phi(c,1,p)=0$ in the interior of $R_2$ then
$$
|\phi(c,1,p)|^2=\left|\frac{1}{2}c^3+ce^{i\theta}(4-c^2)-\frac{3}{2}ce^{2i\theta}(4-c^2)\right|^2=0
$$
and hence
\begin{eqnarray*}
\frac{1}{2}c^3+cp(4-c^2)-\frac{3}{2}c(4-c^2)(2p^2-1)&=&0 ~\mbox{ and } \\ c(4-c^2)\sin\theta-\frac{3}{2}c(4-c^2)\sin2\theta &=&0.
\end{eqnarray*}
On further simplification, this reduces to
$$
\frac{1}{2}c^2+p(4-c^2)-\frac{3}{2}(4-c^2)(2p^2-1)=0 \quad\mbox{and}\quad 1-3p=0,
$$
which is equivalent to $p=1/3$ and $c^2=6$. This contradicts the range of $c\in (0,2)$.  Thus $\phi(c,1,p)\ne 0$ in the interior of $R_2$.

Next, we find the maximum value $F(c,1,p)$ in the interior of $R_2$. Suppose that $F(c,1,p)$  has the maximum at an interior point of $R_2$. Then at such point $\frac{\partial F(c,1,p)}{\partial c}=0$  and $\frac{\partial F(c,1,p)}{\partial p}=0$. From $\frac{\partial F(c,1,p)}{\partial p}=0$ (for points in the interior of $R_2$), a straight forward calculation gives
\begin{equation}\label{p7-050}
p=\frac{2 \left(c^2-3\right)}{3 c^2}.
\end{equation}
Substituting the value of $p$ as given in (\ref{p7-050}) in the relation $\frac{\partial F(c,1,p)}{\partial c}=0$ and further simplification gives
$$
2c-3c^3+ \sqrt{6(c^2+2)}=0.
$$
Taking the last term on the right hand side and squaring on both the sides yields
\begin{equation}\label{p7-052}
9 c^6-12 c^4-2 c^2-12=0.
\end{equation}
This equation has exactly one root in $(0,2)$ which can be shown using the well-known Strum theorem for isolating real roots and hence for the sake of brevity we omit the details. By solving the equation (\ref{p7-052}) numerically, we obtain the approximate root $1.3584$ in $(0,2)$  and the corresponding value of $p$ obtained from (\ref{p7-050}) is $-0.4172$. Thus the extremum points of $F(c,1,p)$ in the interior of $R_2$ lie in a small neighborhood of the points $A_1=(1.3584,1,-0.4172)$ (on the plane $r=1$). Clearly $F(A_1)=9.3689$. Since the function $F(c,1,p)$ is uniformly continuous on $R_2$, the value of $F(c,1,p)$ would not vary too much in the neighborhood of the point $A_1$.

Next we find the maximum value of $F(c,1,p)$ on the boundary of $R_2$. Clearly, $F(0,1,p)=0$, $F(2,1,p)=8$,
$$
F(c,1,-1)=
\begin{cases}
2c+ c(10-3c^2)& \mbox{for}\quad 0\le c\le \sqrt{\frac{10}{3}}\\[2mm]
2c-c(10-3c^2)& \mbox{for}\quad \sqrt{\frac{10}{3}}< c\le 2
\end{cases}
$$
and
$$
F(c,1,1)=
\begin{cases}
2c+c(2-c^2)& \mbox{for}\quad 0\le c\le \sqrt{2}\\[2mm]
2c-c(2-c^2)& \mbox{for}\quad \sqrt{2}< c\le 2.
\end{cases}
$$
By using elementary calculus we find that
$$
\max_{0\le c\le 2}F(c,1,-1)= F\left(\frac{2\sqrt{3}}{3},1,-1\right)=\frac{16\sqrt{3}}{3}=9.2376\quad\mbox{ and }
$$
$$
\max_{0\le c\le 2}F(c,1,1)= F\left(\frac{2\sqrt{3}}{3},1,1\right)=\frac{16\sqrt{3}}{9} =3.0792.
$$
Therefore
$$
\max_{(c,p)\in R_2} F(c,1,p) \approx 9.3689.
$$

On the face $p=-1$,
$$
F(c,r,-1)=
\begin{cases}
\psi(c,r)+\eta_1(c,r) & \mbox{ for }\quad \eta_1(c,r)\ge 0\\[2mm]
\psi(c,r)-\eta_1(c,r)& \mbox{ for }\quad \eta_1(c,r)< 0,
\end{cases}
$$
where $\eta_1(c,r)=c^3 (3r^2+2r+1)-4cr(3r+2)$ and $(c,r)\in R_3:=[0,2]\times[0,1]$. To find the maximum value of $F(c,r,-1)$ in the interior of $R_3$ we need to solve the pair of equations $\frac{\partial F(c,r,-1)}{\partial c}=0$ and $\frac{\partial F(c,r,-1)}{\partial r}=0$ in the interior of $R_3$. But it is important to note that $\frac{\partial F(c,r,-1)}{\partial c}$ and $\frac{\partial F(c,r,-1)}{\partial r}$ may not exist at points in $S_1=\{(c,r)\in R_3: \eta_1(c,r)=0\}$. Solving these pair of equations, we find that
\begin{eqnarray*}
\max_{(c,r)\in  {\rm int \,}R_3\setminus S_1} F(c,r,-1) &=& F\left(\frac{1}{3}(\sqrt{82}-8),\frac{1}{57}(\sqrt{82}-5),-1\right)\\
&=&\frac{4}{81} (41 \sqrt{82}-121)=12.359.
\end{eqnarray*}
Now we find the maximum value of $F(c,r,-1)$ on the boundary of $R_3$ and on the set $S_1$. Note that
$$
\max_{(c,r)\in S_1} F(c,r,-1)\le \max_{(c,r)\in R_3} \psi(c,r)=\frac{37}{3}=12.33.
$$
On the other hand by using elementary calculus, as before, we find that
\begin{align*}
& \max_{0\le r \le 1} F(0,r,-1)=\max_{0\le r \le 1} 12(1-r^2)=F(0,0,-1)=12,\quad \max_{0\le r \le 1} F(2,r,-1)=8,\\
& \max_{0\le c \le 2} F(c,0,-1)=\max_{(c,p)\in R_2} F(c,0,p)= F\left(\frac{2}{3} (3-\sqrt{6}),0,-1\right)=\frac{8}{9} \left(9+\sqrt{6}\right)=12.3546\\
& \mbox{ and } \quad \max_{0\le c \le 2} F(c,1,-1)=F\left(\frac{2\sqrt{3}}{3},1,-1\right)=\frac{16\sqrt{3}}{3}=9.2376.
\end{align*}
%Thus
%$$
%\max_{(c,r)\in \partial R_3} F(c,r,-1)=F\left(\frac{2}{3} (3-\sqrt{6}),0,p\right)=\frac{8}{9} \left(9+\sqrt{6}\right)=12.3546,
%$$
%where $\partial R_3$ denotes the boundary of $R_3$.
Hence, by combining the above cases we obtain
\begin{eqnarray*}
\max_{(c,r)\in  R_3} F(c,r,-1)&=&F\left(\frac{1}{3}(\sqrt{82}-8),\frac{1}{57}(\sqrt{82}-5),-1\right)\\
&=&\frac{4}{81} (41 \sqrt{82}-121)=12.359.
\end{eqnarray*}

On the face $p=1$,
$$
F(c,r,1)=
\begin{cases}
\psi(c,r)+\eta_2(c,r) & \mbox{ for }\quad \eta_2(c,r)\ge 0\\[2mm]
\psi(c,r)-\eta_2(c,r)& \mbox{ for }\quad \eta_2(c,r)< 0,
\end{cases}
$$
where $\eta_2(c,r)=c^3 (3r^2-2r+1)-4cr(3r-2)$ for $(c,r)\in R_3$. To find the maximum value of $F(c,r,1)$ in the interior of $R_3$ we need to solve the pair of equations $\frac{\partial F(c,r,1)}{\partial c}=0$ and $\frac{\partial F(c,r,1)}{\partial r}=0$ in the interior of $R_3$. But it is important to note that $\frac{\partial F(c,r,1)}{\partial c}$ and $\frac{\partial F(c,r,1)}{\partial r}$ may not exist at points in $S_2=\{(c,r)\in R_3: \eta_2(c,r)=0\}$. Solving these pair of equations, we find that
\begin{eqnarray*}
\max_{(c,r)\in  {\rm int \,}R_3\setminus S_2} F(c,r,1)&=&F\left(\frac{1}{3}(8-\sqrt{46}),\frac{1}{75}(11-\sqrt{46}),1\right)\\
&=&\frac{4}{81}(95+23 \sqrt{46})=12.3947.
\end{eqnarray*}
Now, we find the maximum value of $F(c,r,1)$ on the boundary of $R_3$ and on the set $S_2$. By noting that (see earlier cases)
\begin{align*}
& \max_{(c,r)\in S_2} F(c,r,1)\le \max_{(c,r)\in R_3} \psi(c,r)=\frac{37}{3}=12.33,\\
& \max_{0\le r \le 1} F(0,r,1)=12,\quad \max_{0\le r \le 1} F(2,r,1)=8,\\
& \max_{0\le c \le 2} F(c,0,1)=\frac{8}{9} \left(9+\sqrt{6}\right)=12.3546,\\
& \max_{0\le c \le 2} F(c,1,1)=\frac{16\sqrt{3}}{9} =3.0792
\end{align*}
and combining all the cases, we find that
\begin{eqnarray*}
\max_{(c,r)\in  R_3} F(c,r,1)&=&F\left(\frac{1}{3}(8-\sqrt{46}),\frac{1}{75}(11-\sqrt{46}),1\right)\\
&=&\frac{4}{81}(95+23 \sqrt{46})=12.3947.
\end{eqnarray*}

Let $S'=\{(c,r,p)\in R: \phi(c,r,p)=0\}$. Then
$$
\max_{(c,r,p)\in S'} F(c,r,p)\le \max_{(c,r)\in R_3} \psi(c,r)=\psi\left(0,\frac{1}{3}\right)=\frac{37}{3}=12.33.
$$
We prove that $F(c,r,p)$ has no maximum value at any interior point of $R\setminus S'$. Suppose that $F(c,r,p)$  has a maximum value at an interior point of $R\setminus S'$. Then at such point $\frac{\partial F(c,r,p)}{\partial c}=0$, $\frac{\partial F(c,r,p)}{\partial r}=0$ and $\frac{\partial F(c,r,p)}{\partial p}=0$. Note that $\frac{\partial F(c,r,p)}{\partial c}$, $\frac{\partial F(c,r,p)}{\partial r}$ and $\frac{\partial F(c,r,p)}{\partial p}$ may not exist at points in $S'$. From  $\frac{\partial F(c,r,p)}{\partial p}=0$ (for points in the interior of $R\setminus S'$), a straight forward but laborious calculation gives
\begin{equation}\label{p7-060}
p= \frac{3 c^2 r^2+c^2-12 r^2}{6 c^2 r}.
\end{equation}
Substituting the value of $p$ as given in  (\ref{p7-060}) in  $\frac{\partial F(c,r,p)}{\partial r}=0$ and simplifying, we obtain
$$
(4-c^2)r( \sqrt{6(c^2+2)}-6)=0.
$$
This equation has no solution in the interior of $R\setminus S'$ and hence $F(c,r,p)$ has no maximum in the interior of $R\setminus S'$.

Thus combining all the above cases we find that
\begin{eqnarray*}
\max_{(c,r,p)\in  R} F(c,r,p)&=&F\left(\frac{1}{3}(8-\sqrt{46}),\frac{1}{75}(11-\sqrt{46}),1\right)\\
&=&\frac{4}{81}(95+23 \sqrt{46})=12.3947
\end{eqnarray*}
and hence from (\ref{p7-042}) we obtain
\begin{equation}\label{p7-065}
|\gamma_3|\le \frac{1}{972}(95+23 \sqrt{46})=0.2582.
\end{equation}

We now show that the inequality (\ref{p7-065}) is sharp. An examination of the proof shows that equality holds in (\ref{p7-065}) if we choose $b_3=1$, $c_1=c=\frac{1}{3}(8-\sqrt{46})$, $x=\frac{1}{75}(11-\sqrt{46})$ and $t=1$ in (\ref{p7-035}). For such values of $c_1, x$ and $t$,  Lemma \ref{p7-lemma010} gives  $c_2=\frac{1}{27}(134-19 \sqrt{46})$ and $c_3=\frac{2}{243}(721-71\sqrt{46})$. A function $P\in\mathcal{P}$ having the first three coefficients $c_1, c_2$ and $c_3$ as above is given by
\begin{align}\label{p7-140}
P(z) &=(1-2\lambda) \frac{1+z}{1-z}+\lambda \frac{1+uz}{1-uz}+\lambda \frac{1+\overline{u}z}{1-\overline{u}z}\\
& =1+\frac{1}{3}(8-\sqrt{46})z +\frac{1}{27}(134-19 \sqrt{46})z^2 +\frac{2}{243}(721-71\sqrt{46})z^3+\cdots,\nonumber
\end{align}
where $\lambda=\frac{1}{10}(-4+\sqrt{46})$ and $u=\alpha+i\sqrt{1-\alpha^2}$ with $\alpha=\frac{1}{18}(-1-\sqrt{46})$. Hence the equality holds in (\ref{p7-065}) for a function $f$ defined by $zf'(z)=g(z)P(z)$, where $g(z)=z/(1-z^2)$ and $P(z)$ is given by (\ref{p7-140}). This completes the proof.
\end{proof}

%\vspace{4mm}
\noindent\textbf{Acknowledgement:}
The authors thank Prof. K.-J. Wirths for useful discussion and suggestions.
The first author thank University Grants Commission for the financial support through UGC-SRF Fellowship. The second author thank SERB (DST) for financial support.

\end{document}